\documentclass{amsart}
\usepackage{amssymb}
\usepackage{amsthm}
\usepackage{indentfirst}
\usepackage{amsmath}
\usepackage{amscd}
\numberwithin{equation}{section}

\newtheorem{Theorem}{Theorem}[section]
\newtheorem{Definition}{Definition}[section]
\newtheorem{Proposition}[Theorem]{Proposition}
\newtheorem{Lemma}[Theorem]{Lemma}
\newtheorem{Notation}[Theorem]{Notation}

\newtheorem{Assumption-Notation}[Theorem]{Assumption-Notation}

\newtheorem{Remark}[Theorem]{Remark}
\newtheorem{Corollary}[Theorem]{Corollary}
\newtheorem{Problem}{Problem}

\newtheorem{Claim}[Theorem]{Claim}
\newtheorem*{Fact}{Fact}

\newtheorem*{Acknowledgments}{Acknowledgments}

\begin{document}

\title[Surfaces with $p_g = q= 1$ and $K^2 = 7$]{surfaces with $p_g = q= 1$, $K^2 = 7$ and non-birational bicanonical mpas}
\address{Lei Zhang\\School of Mathematics Sciences\\Peking University\\Beijing 100871\\P.R.China}
\email{lzhpkutju@gmail.com}
\author{Lei Zhang}
\maketitle

\textbf{Abstract} Let $S$ be a minimal surface of general type
with $p_g = q = 1, K_S^2 = 7$. We prove that the degree of the
bicanonical map is 1 or 2. Furthermore, if the degree is 2, we describe $S$ by a double cover.

\section{Introduction}
For the classification of surfaces of general type, it is an
effective way to consider their bicanonical maps, especially when $p_g$ or $\chi$ is small. Recall that the bicanonical map of a surface $S$ is a map $\phi: S \dashrightarrow \mathbb{P}^{\chi(S) + K_S^2 - 1}$ defined by the linear system $|2K_S|$. It is a morphism if $K_S^2 \geq 5$ due to Reider (cf. \cite{Re}), and the image is a surface if $K_S^2 \geq 2$ due to Xiao (cf. \cite{X2}. Many people were devoted to studying the bicanonical maps of surfaces. We refer the readers to \cite{BCP} for a survey of the known results. Up to now, for the surfaces with non-birational bicanonical map, only the case when $p_g = q \leq 1$ is not clear.

For surfaces with $p_g = q = 0$ and non-birational bicanonical map, when $K^2 \geq 6$, Mendes Lopes and Pardini got a complete classification, we refer the readers to the papers
\cite{MP2}, \cite{MP3} and \cite{Par} for their main results; when $K^2 = 5$, reads can refer \cite{Zh1} and \cite{Zh2} for the recent results. Here we
mention that for a surface $S$ with $K_S^2 = 7,8$, in \cite{MP1}, the authors
proved: the bicanonical map is a
degree 2 map to a rational surface, i.e., $S$ is a smooth minimal model of a du Val double plane (see Section \ref{deq2} for the definition); moreover when $K_S^2 = 8$, $S$ is isogenous to a product of two curves.

For a surface $S$ with $p_g(S) = q(S) = 1$, recall
that $K_S^2 \geq 2$ by \cite{Bom} and $K_S^2 \leq 9$ by \cite{Miy} or
\cite{Yau}. When
$K_S^2 = 2,3$, Catanese and Ciliberto gave a complete classification (cf. \cite{Cat}, \cite{CC1}, \cite{CC2}). When $K_S^2$ is big, only the cases when $K_S^2=9,8$ are well understood from the point of view of the non birationality of bicanonical map. In the first case, the bicanonical map is
birational. And in the latter case, under the condition that the bicanonical map factors through a double cover onto a rational surface, Polizzi proved that $S$ is
isogenous to a product of two curves (cf. \cite{Pol}); and Borreli proved it is the case as in \cite{Pol} indeed, hence the bicanonical map is of degree 2 and factors through a double cover over a
rational surface (cf. \cite{Bor2}).
The results are analogous to the cases when $p_g = q = 0,K^2 = 9,8$. Now what
is the case when $K^2 = 7$? In this paper, we proved:

\begin{Theorem}\label{main}
Let $S$ be a smooth minimal surface of general type with $p_g(S) = q(S) = 1,K_S^2
= 7$; denote by $\phi: S \rightarrow \mathbb{P}^7$ its bicanonical
map and by $\Sigma$ the bicanonical image. If $\phi$ is not
birational, then it is of degree 2, and
$S$ is the smooth minimal model of a double cover $X$ over a surface $Y$ branched along a curve $B$ where $Y$ and $B$ satisfy one of the following:
\begin{enumerate}
\item[(I)]{$Y$ is a Kummer surface, and if we denote by  $C_i,i=1,...,16$ the 16 disjoint $(-2)$-curves, then $B = B' + \sum_{i=1}^{i=16}C_i$ where $B'$ does not intersect any one of the $C_i$'s and has a [3,3]-point as its unique singularity, precisely
    \begin{enumerate}
    \item[(i)]{either $Y$ is the minimal resolution of a quartic surface in $\mathbb{P}^3$ with 16 nodes, and $B' \equiv 2H$ where $H$ is the pull-back of a hyperplane section;}
    \item[(ii)]{or $Y$ is the minimal resolution of a double cover of $\mathbb{P}^1\times \mathbb{P}^1$ branched over 4 vertical curves and 4 horizontal curves, and $B'$ is linearly equivalent to the pull-back of a $(2,2)$-divisor on $\mathbb{P}^1\times \mathbb{P}^1$.}
    \end{enumerate}}
\item[(II)]{$Y = \mathbb{P}^2$, and $B = G_5 + L_1 + L_2+ L_3+L_4 + L_5$ where $G_5$ is a curve of degree $15$ and the $L_i$'s are distinct lines through a point $\gamma$. The essential singularities of $B$ are as following:
    \begin{enumerate}
    \item[-]{$\gamma$ is a singular point of multiplicity $12$,}
    \item[-]{a [5,5]-point $p_i$ on $L_i,i=1,...,5$,}
    \item[-]{a [3,3]-point $q$.}
    \end{enumerate}
Furthermore there is exactly one conic passing through the points $p_i,i=1,...5$ and $q$.}
\end{enumerate}
Conversely, given a surface $Y$ and a curve $B$ satisfying $(I)$ or $(II)$, we get a surface $S$ with the given invariants and bicanonical map of degree 2.
\end{Theorem}

\begin{Remark}
The surfaces of type $(II)$ do exist: an example is constructed by Rito who used computer to find a
proper branch divisor, we refer the readers to \cite{Rito2} for
the details. But we still lack of examples for Case $(I)$. The answer is probably negative due to the analysis in Section \ref{K3}
\end{Remark}

The plan of the paper is as follows: in Section \ref{bt}, we list the tools used in this paper; in Section \ref{dneq4}, we prove the degree of $\phi$ cannot be 4, thus it is of degree 2; in Section \ref{deq2}, we study the case when
$\phi$ is a degree 2 map and proved that the bicanonical image is birational to either a rational surface or a Kummer surface arising from a polarized Abelian variety; then finally in Section \ref{K3}, we discuss the existence of the surfaces of type $(I)$.

When we studied the degree of $\phi$, we notice that since $p_g(S) = 1$, the unique effective canonical divisor is nearly ``determined'', then we got much information using double cover trick and ramification formula. This method also applies when we consider the case $K^2=8$, so we give an alternative simpler proof of the results in \cite{Bor2} (cf. Remark \ref{bor}). We guess that all the minimal surfaces with $K^2 = 6, p_g(S) = q(S) = 1$ have bicanonical maps of degree $\leq 2$ since it is true for all the known examples. Unfortunately our method will not apply to this case.

When we consider the case $d=2$ in Section \ref{deq2}, the method is routine. And in the last section, we first provide a method to construct our surfaces, then reduce the existence problem to whether two certain sets intersect in the Hilbert scheme of the quartics in $\mathbb{P}^3$.

\textbf{Notations and conventions:} We work over complex numbers;
all varieties are assumed to be compact and algebraic. We don't
distinguish between the line bundles and the divisors on a smooth
variety, and we use both the additive and multiplicative notation.
We say a line bundle is effective if it has a global non-zero
section. Let $D_1$ and $D_2$ be two Cartier divisors. Then $D_1 \equiv D_2$ means that
the two divisors are linearly equivalent. We say a divisor $D$ is a $\mathbb{Q}$-Cariter Weil divisor if it is a Weil divisor such that $mD$ is Cartier for some positive integer $m$. We say that a curve singularity is negligible if it is either a
double point or a triple point which resolves to at most a double point after
one blow-up. A $[n, n]$-point is a point of multiplicity $n$
with an infinitely near point also of multiplicity $n$.
\begin{Acknowledgments} I am grateful to Prof. Jinxing Cai and Wenfei Liu for their valuable suggestions and many useful discussions
during the preparation of this paper. Borrelli informed me that under the assumption that the bicanonical map factors through a degree 2 map, he got some results much earlier. Here I thank him for he allowed me to submit this paper.
\end{Acknowledgments}

\section{Basic tools}\label{bt}

\subsection{Double cover formula}
Let $S$ be a smooth surface, let $D \subset S$ be a curve with at most negligible singularities
(possibly empty) and let $M$ be a line bundle on $S$ such that $2M
\equiv D$. Then there exists a double cover $\pi: Y \rightarrow S$
branched along $D$ and such that $\pi_*\mathcal{O}_Y =
\mathcal{O}_S \bigoplus M^{-1}$.  Note
that $Y$ is smooth if $D$ is smooth, $Y$ has canonical singularities if $D$ has negligible singularities, and $Y$ is connected if and
only if $M$ is non-trivial. The invariants of Y can be
calculated as follows:
\begin{equation}\label{tool}
\begin{split}
K_Y^2 &= 2(K_S + M)^2 \\
\chi(\mathcal{O}_Y) &= 2 \chi(\mathcal{O}_S)+ \frac{1}{2}M(K_S +
M)\\
p_g(Y) &= p_g(S) + h^0(S, K_S + M)
\end{split}
\end{equation}

\subsection{Results of fibrations}
\begin{Proposition}\label{minfib} (cf. \cite{BPV} or \cite{Beau})
Let $S$ be a smooth surface and $f: S \rightarrow B$ a relatively
minimal fibration on to a curve of genus $b$, and let $F$ be a general
fiber, and assume its genus is $g$. Then
\begin{enumerate}
\item[(0)]{$\chi(X) \geq (b-1)(g-1)$;}
\item[(i)]{$K_S^2 \geq 8(b-1)(g-1)$;}
\item[(ii)]{$e(S) \geq 4(b-1)(g-1)$;}
\item[(iii)]{$q(S) \leq b+g$.}
\end{enumerate}
If equality holds in $(i)$, then the fibers of $f$ have constant
modulus; if equality holds in $(ii)$, then every fiber of f is smooth;
and if equality holds in $(iii)$, then $S \cong F \times B$.
\end{Proposition}

We will use the direct image of the relative canonical sheaf of a
fibration, recall that:

\begin{Proposition}[cf. \cite{Fuj}]
Let $f: S \rightarrow C$ be a fibration of $g\geq 1$, and let
$\omega_{S/C}$ be the relative canonical line bundle. Then $f_*\omega_{S/C}$
is semipositive vector bundle on $C$.
\end{Proposition}

The following lemma is well known to experts, but we give detailed proof
here for lack of references.
\begin{Lemma}\label{cvroffbr}
Let $f:S \rightarrow \Delta$ be a fibration over the unit disc $\Delta$ such that $S_0 = 2M$ be the only singular fiber. Let $\pi: X \rightarrow S$ be an etale double cover given by the relation $2L \equiv \mathcal{O}_S$. Let $\Delta' \rightarrow \Delta$ be a double cover given by $t \rightarrow s^2$ branched along the point $0$. If $X$ coincides with the normalization of the fiber product $\Delta' \times_{\Delta}S$, then $L \equiv M$.
\end{Lemma}
\begin{proof}
 Denote by $\pi': X' \rightarrow S$ the double cover given by the relation $2M \equiv \mathcal{O}_S$. Considering the Stein factorizations of the composition maps $f\circ \pi'$ and $f\circ \pi$, we find that both $X$ and $X'$ are isomorphic to the normalization of the product $S\times _{\Delta}\Delta'$, hence the double cover $\pi$ coincides with $\pi'$, and thus $\pi_*\mathcal{O}_X \cong \pi_*\mathcal{O}_{X'}$. So the lemma follows from the fact that $\pi_*\mathcal{O}_X \cong \mathcal{O}_S \oplus \mathcal{O}_S(L^{-1})$ and $\pi_*\mathcal{O}_X \cong \mathcal{O}_S \oplus \mathcal{O}_S(M^{-1})$.
\end{proof}

\begin{Lemma}\label{basiclemma}
Let $f: S \rightarrow \mathbb{P}^1$ be a fibration of  genus $3$;
let $\pi: Y \rightarrow S$ be an etale double cover; and let $\pi'
\circ g: Y \rightarrow C \rightarrow \mathbb{P}^1$ the Stein
factorization of the map $f\circ \pi$. Suppose that $q(Y) = 3$,
$p_g(S) = q(S) = 1$ and $p_g(Y) > p_g(S)$. Then $g(C) \geq 1$ and $\pi'$ is a double
cover.
\end{Lemma}
\begin{proof}
If $g(C) \geq 1$, then we are done. So by contrary, suppose that $C$ is a rational curve.

First we claim that $\pi'$ is an isomorphism. Indeed, otherwise the general
fiber of $g: Y \rightarrow C$ is of genus $3$, but then since $q(Y) = 3$, Lemma
\ref{minfib} gives that $Y \cong F \times
\mathbb{P}^1$ which contradicts $p_g(Y) > 0$.

Therefore $g = f\circ \pi$, and it is a genus $5$ fibration. Let
$\sigma$ be the involution induced by $\pi$. Since $\omega_Y$ is
a $\sigma$-sheaf and $\sigma$ is compatible with the fibration,
$g_*\omega_Y$ is also a $\sigma$-sheaf. We can write $g_*\omega_Y = V^+
\oplus V^-$ where $V^+$ (resp.$V^-$) is a
$\sigma$-invariant(resp.anti-invariant) vector bundle on
$\mathbb{P}^1$ of rank 3 (resp.2). Consider the Lerray spectral sequence
$$0 \rightarrow H^0(\mathcal{O}_{\mathbb{P}^1}, g_*\omega_Y) \rightarrow
H^0(Y, \omega_Y) \rightarrow H^1(\mathcal{O}_{\mathbb{P}^1}, g_*\omega_Y)
\rightarrow H^1(Y, \omega_Y) \rightarrow
H^0(\mathcal{O}_{\mathbb{P}^1}, R^1g_*\omega_Y)$$  By relative
duality, we have $R^1g_*\omega_{Y/C} \cong (g_*\mathcal{O}_Y)^* \cong
\mathcal{O}_{\mathbb{P}^1}$, hence $R^1g_*\omega_Y \cong
\mathcal{O}_{\mathbb{P}^1}(-2)$. So the last term of the Lerray sequence is 0, and we get
$H^0(Y,\omega_Y) \cong H^0(\mathcal{O}_{\mathbb{P}^1}, g_*\omega_Y)$ and
$H^1(Y,\omega_Y) \cong H^1(\mathcal{O}_{\mathbb{P}^1}, g_*\omega_Y)$. Therefore, we have $p_g(Y) - p_g(S) = h^0(\mathcal{O}_{\mathbb{P}^1}, V^-)$ and
$q(Y) - q(S) = h^1(\mathcal{O}_{\mathbb{P}^1}, V^-)$.

Remark that if we assume $g_*\omega_Y \cong \oplus_{i = 1} ^{i = g}
\mathcal{O}_{\mathbb{P}^1}(a_i)$, since $g_*\omega_Y \cong g_*\omega_{Y/C}
\otimes \mathcal{O}_{\mathbb{P}^1}(-2)$, then semi-positivity of
$g_*\omega_{Y/C}$ gives that $a_i \geq -2$. Now we assume $V^- \cong
\mathcal{O}_{\mathbb{P}^1}(a) \oplus
\mathcal{O}_{\mathbb{P}^1}(b)$, then $a, b \geq -2$. By $q(Y)
- q(S) = 2$, we have $h^1(\mathcal{O}_{\mathbb{P}^1}, V^-) = 2$,
thus $a = b = -2$. So it follows that
$h^0(\mathcal{O}_{\mathbb{P}^1}, V^-) = 0$, which implies that $p_g(Y) =
p_g(S)$, and we get a contradiction from the hypothesis.
\end{proof}

\subsection{Known results on surfaces with $p_g= q =1$ and $K^2 =
7$}

\begin{Proposition} \label{sing}
Let $S$ be a smooth minimal surface of general type with $p_g = q = 1,K^2
= 7$ and $\bar{S}$ be its canonical model. Then $\bar{S}$ has at most
one singularity: an $A_1$-singularity or an $A_2$-singularity. In
particular if $Z$ is a reduced divisor supported on $(-2)$-curves,
then $Z^2 = -2$.
\end{Proposition}
\begin{proof}
Let $k$ be the maximal number of disjoint $(-2)$-curves on $S$. By
\cite{Miy}, we have $\frac{3}{2}k \leq c_2(S) - \frac{1}{3}K_S^2$,
i.e., $k \leq \frac{16}{9}$, thus $k=1$. Then the proposition is an easy consequence.
\end{proof}

By results of \cite{X1}, we have
\begin{Proposition}\label{ggeq3}
Let $S$ be a smooth minimal surface of general type with $p_g = q = 1,K^2
= 7$. Then $S$ has no genus 2 fibration.
\end{Proposition}

\subsection{A property of the pull-back of a divisor with negative
self-intersection number}
\begin{Lemma}\label{rmfdivisor}
Let $h: X \rightarrow T$ be a generically finite morphism between
two normal surfaces. Let $e \subset T$ be a reduced and irreducible $\mathbb{Q}$-Cartier Weil
divisor such that $e^2 < 0$. Denote by $R$ be the ramification divisor,
and let $R'$ be an effective divisor such that $R' \leq R$. Then we have
$R'(h^*e)
> (h^*e)^2$.
\end{Lemma}
\begin{proof}
Let $g \circ \eta: X \rightarrow Y \rightarrow T$ be the Stein
factorization of $h$. Write $g^*e = \sum_ia_iE_i$. Note that
$(g_*E_i)e < 0$ since $e^2 <0$. Then we have
\begin{equation}
\begin{split}
R'(h^*e)
&= (\eta_* R')( g^*e )\geq (\sum_i(a_i - 1)E_i)(g^*e)\\
&= (g^*e - \sum_iE_i)g^*e \\
&= (g^*e)^2 - \sum_i(g_*E_i)e >(g^*e)^2 = (h^*e)^2
\end{split}
\end{equation}
\end{proof}

\section{the degree of the bicanonical map}\label{dneq4}
This section is devoted to prove that the degree of the bicanonical map is 1 or 2.

\subsection{The bicanonical image}\label{image}
\begin{Notation}\label{notation}
Let $S$ and $\phi: S \rightarrow \Sigma \subset \mathbb{P}^7$ be as in Theorem \ref{main}, and denote by $d$ the degree of $\phi$. Denote by $\bar{S}$ the canonical model of $S$.
So $\phi$
factors though a map $\bar{\phi}: \bar{S} \rightarrow \Sigma$
which is a finite morphism.
\end{Notation}

Note that $deg(\Sigma)\geq 6$ since $\Sigma$ is not contained in a hyperplane. Then we have $d = \frac{(2K_S)^2}{deg(\Sigma)} \leq 4$, thus $d = 1,2$ or $4$. We will prove $d \neq 4$ by contradiction. From now on, we assume the degree of $\phi$ is 4, thus $\Sigma$ is a linearly normal surface of degree 7 in
$\mathbb{P}^7$.

By Theorem 8 in \cite{Na}, $\Sigma$ is one of the
following:
\begin{enumerate}
\item[$\bullet$]{A cone over an elliptic curve of degree 7 in $\mathbb{P}^6$;}
\item[$\bullet$]{the image of $\psi:
\hat{P} \rightarrow \mathbb{P}^7$, where $\hat{P}$ is the blow-up
of $\mathbb{P}^2$ at two points $P_1, P_2$ and $\psi$ is given by
the linear system $|-K_{\hat{P}}|$.}
\end{enumerate}
First remark that the map $f$ in the proof of Lemma 1.5 of \cite{Bor2}
factors through a morphism $\bar{f}: \bar{S} \rightarrow C$ since $g(C)
\geq 1$. The argument applied to $\bar{f}: \bar{S} \rightarrow C$ shows that the first case does
not occur.

So $\Sigma$ falls into the latter case. Denote by $\rho: \hat{P}\rightarrow \mathbb{P}^2$ the blow-up at two points $P_1, P_2$, by $l$ the pull-back of
a general line on $\mathbb{P}^2$, by $e_i$ the exceptional divisor
over $P_i$, by $c$ the strict transform of the line through $P_1,
P_2$. According to whether $P_1$ and $P_2$ are infinitely
near, we get the two cases:\\
$Case~A:$ $P_1$ and $P_2$ are distinct, thus $\Sigma \simeq
\hat{P}$ and $-K_\Sigma \equiv 3l - e_1 - e_2$.\\
$Case~B:$ $P_2$ is infinitely near to $P_1$, thus $-K_{\hat{P}}
\equiv 3l - e_1 - 2e_2$, and the map $\psi: \hat{P} \rightarrow
\Sigma$ contracts $e_1$.

Abusing notations, we also denote the two divisors $\psi_*l$ and
$\psi_*c$ by $l$ and $c$ on $\Sigma$.

\subsection{The pull-back $\phi^*c$}\label{pullbackofc}
First we consider the pull-back $\phi^*c$.
\begin{Proposition}\label{c}
We can write $\phi^*c = 2(C' + Z')$ where $C'$ is reduced and
irreducible with $K_S C' = 1$ and $Z'$ is zero or supported on
some $(-2)$-curves.
\end{Proposition}
\begin{proof}
Write $\phi^*c = C + Z$ where $C$ is the strict transform of $c$ with respect to $\phi$ and $Z$
is zero or supported on some $(-2)$-curves. First we claim:
\begin{Claim}Suppose that there exists a line bundle $L$ on $S$ such that $\phi^*c \equiv 2L$ and that $C$ is reduced. Then $\phi^*c$ is as one of
the following:
\begin{enumerate}
\item[(i)]{$\phi^*c = C$ is a smooth rational $(-4)$-curve;}
\item[(ii)]{$\phi^*c = A + B$ where $A,B$ are rational
$(-3)$-curves with $AB=1$;}
\item[(iii)]{$\phi^*c = A + B + Z$ where $A,B$ are rational
$(-3)$-curves and $Z$ is a reduced divisor supported on $(-2)$-curves such that
$AB = 0, AZ=BZ =1$.}
\end{enumerate}
In particular, the divisor $\phi^*c$ is reduced and has at most nodes.
\end{Claim}
\begin{proof}[Proof of the claim]
Note that $C^2 = (\phi^*c - Z)^2 = (\phi^*c)^2 + Z^2 \leq -4$ and the
equality holds if and only if $Z = 0$.

If $C$ is irreducible, then by $K_SC = 2$, we conclude that $C^2
\geq -4$, consequently $Z = 0$ and $C$ is a rational $(-4)$-curve. This is $(i)$.

If $C$ is reducible, then $C = A +B$ where $A,B$ are reduced an irreducible divisors with $K_SA = K_SB = 1$,
hence $A^2 \geq -3, B^2 \geq -3$ and $C^2 = -4$ or $-6$. Noticing that $(\phi^*c)A$ and
$(\phi^*c)B$ are negative and even integers since $\phi^*c \equiv 2L$, we get
the following possibilities: \\
\begin{enumerate}
\item[(ii)] {$C^2 = -4,AB = 1, A^2= B^2= -3, Z=0$;}
\item[(iii)] {$C^2 = -6,AB = 0, A^2 = B^2 = -3, AZ= BZ = 1, Z^2 = -2$}
\end{enumerate}
Furthermore, in the latter case, $Z$ is reduced since $Z^2 = -2$.
In either case, the irreducible components of $\phi^*c$ are all smooth
rational curves, and any two components intersect transversely if
they do. So the claim is true.
\end{proof}

Now let's continue the proof of the proposition.
Since $-K_\Sigma \equiv 2l + c$, we can write $2K_S \equiv 2L +
\phi^*c$. If $C$ is reduced, then $\phi^*c$ is reduced, hence the relation $2(K_S - L) \equiv
\phi^*c$ gives a double cover $\pi: Y \rightarrow S$. The claim above
guarantees that $Y$ has at most canonical singularities. We claim that $Y$ is minimal. Indeed, since $-K_{\Sigma} + c$ is nef, so is $\pi^*(2K_S + \phi^*c)$, then the formula $2K_Y \equiv \pi^*(2K_S + \phi^*c)$ gives that $K_Y$ is also nef, thus $Y$ is minimal.

Let $X \rightarrow Y$ be the minimal resolution. Then $X$ is
minimal. Abusing notations, we also denote the composed map $X
\rightarrow Y \rightarrow S$ by $\pi$. By use of formula
\ref{tool}, $X$ has the following invariants
$$\chi(X) = 2,~p_g(X) = h^0(S, 2K_S - L) + p_g(S) = 5, ~q(X) = 4,~K_X^2 = 16$$

Then arguing as in the proof of Lemma 1.6 in \cite{Bor2}, we get a fibration $f:X
\rightarrow B$ where $B$ is a curve of genus $\geq 2$, a fibration $g: S \rightarrow C$ and a double
cover $\pi': B \rightarrow C$ such that $g\circ \pi =
\pi'\circ f$. Note that the two fibrations $f$ and $g$ have the same genus.
Then by Proposition \ref{ggeq3}, the genus of a general
fiber $F$ of $f$ is $\geq 3$. Applying Proposition
\ref{minfib} $(0)$, i.e., $2 = \chi(X) \geq (1-g(B))(1-g(F))$, we
conclude that $g(F) = 3$, $g(B) = 2$ and the equality holds, thus $16 =K_X^2 =
8(1-g(B))(1-g(F))$. By Noether's formula, we get that $e(X) =
4(1-g(B))(1-g(F))$, so every fiber of $f: X \rightarrow
B$ is smooth by Proposition \ref{minfib}. However, considering the pull-back $\pi^*C$,
a divisor composed with some rational curves, it must be contained in one
fiber, in particular the fiber is singular.

So $C$ is not reduced, thus $C = 2C'$ since $K_SC = 2$. Write
$\phi^*c = 2(C' + Z') + Z''$ where $Z''$ is reduced and supported
on $(-2)$-curves if it is not zero. Since $4|(\phi^*c - 2(C' +
Z'))^2 = Z''^2$, Proposition \ref{sing} implies that $Z'' = 0$,
therefore the proposition is true.
\end{proof}

\begin{Remark}\label{pullback} Let $d$ be a Cartier divisor on $\Sigma$, and denote by $D$
the strict transform of $d$. If $D =2D'$, by similar argument as
above, we can write $\phi^*d = 2(D' + Z')$. In the following, if
this case occurs, by abuse of notations, we write $\phi^*d = 2D'$.
\end{Remark}

\subsection{The effective
divisor linearly equivalent to $K_S$}\label{KS}
In the following, we make an important observation: the effective
divisor linearly equivalent to $K_S$ is nearly determined.

Since $p_g(\bar{S}) = 1$, there exists a unique effective divisor $D
\equiv K_{\bar{S}}$ on $\bar{S}$. And by $|2K_{\bar{S}}| =
\bar{\phi}^*|-K_\Sigma|$, we can find a divisor $d \in
|-K_\Sigma|$ on $\Sigma$ such that $2D = \bar{\phi}^*d$. Write $d
= \sum_i b_id_i$ where the $d_i$'s are reduced and irreducible
divisors, and for every $d_i$, write $\bar{\phi}^*b_id_i =
\sum_j2a_{ij}D_{ij}$ where $D_{ij}$ is reduced and irreducible for
every $j$. Since $\bar{\phi}$ is a finite map, the
pull-backs of two distinct irreducible divisors have no common
components, so the $D_{ij}$'s are distinct.
\begin{Claim}\label{coefficient} At least one of the $b_i$'s is $\geq
2$.\end{Claim}
\begin{proof}
By contrary assume that $b_i = 1$ for every $i$. Then $\bar{\phi}$ is ramified along $D_{ij}$ with branching order $2a_{ij}$. Let $R$ be the ramification
divisor of $\bar{\phi}$. Immediately we have $a_{ij}D_{ij} \leq
(2a_{ij} -1)D_{ij} \leq R$, thus $D = \sum_{ij} a_{ij}D_{ij} \leq
R$. By the formula $D \equiv K_{\bar{S}} \equiv
\bar{\phi}^*K_\Sigma + R$, we get that $R - D \equiv
\bar{\phi}^*(-K_\Sigma)$. Since $R- D$ is effective and
$|2K_{\bar{S}}| = |\phi^*(-K_\Sigma)| = \phi^*|-K_\Sigma|$, we can
find an element $H \in |-K_\Sigma|$ such that $R - D =
\bar{\phi}^*H$. A contradiction follows from the property of the ramification divisor.
\end{proof}

On the other hand, by $\sum_{ij} a_{ij}D_{ij} \geq
\bar{\phi}^*(\sum_i [\frac{b_i}{2}]d_i)$ and $h^0(\bar{S}, D) =
1$, we obtain
\begin{equation}\label{condition}h^0(\Sigma, \sum_i [\frac{b_i}{2}]d_i) = 1
\end{equation} In particular if $b_i \geq 2$, then $d_i^2 < 0$.

Furthermore if $b_i = 1$, we can
write $\bar{\phi}^*d_i = 2\bar{D_i'}$ where $\bar{D_i'}$ is an
effective divisor. So by Remark \ref{pullback}, we have $\phi^*d_i
= 2D_i'$ for some effective divisor $D_i'$ on $S$.

If $\Sigma$ is smooth, i.e., $Case ~A$, by Claim \ref{coefficient}
and condition \ref{condition}, we get the following possibilities
for $D$ and $d$:
\begin{enumerate}
\item[Case $A_1(i)$:] {$d = 2e_i + c + m_1 + m_2$ where $m_1, m_2$ are the
strict transforms of two lines passing through $P_i$,
correspondingly $K_S \equiv \phi^*e_i + C' + M_1' + M_2'$ where
$\phi^*m_j = 2M_j',j = 1,2$ and $\phi^*c = 2C'$;}
\item[Case $A_2$:] {$d = 2c + e_1 + e_2 + l$ where $l$ is the strict
transform of a line not through any one of $P_1, P_2$,
correspondingly $K_S \equiv 2C' + E_1' + E_2' + L'$ where
$\phi^*e_i = 2E_i',i =1,2$ and $\phi^*l = 2L'$.}
\end{enumerate}

\subsection{Proof of $d\neq 4$}\label{pf}
\begin{proof}
We rule out $Case~A_i$, $Case ~B$ case by case.

$Case~A_1(i):$ We consider Case $A_1(1)$, i.e., $d = 2e_1 +
c + m_1 + m_2$. A general curve $L_1 \in\phi^*|l -
e_1|$ is connected because otherwise we will get a fibration of
genus 2 on $S$, so the linear system $|L_1|$ defines a fibration $f: S \rightarrow \mathbb{P}^1$. A general fiber of
$f$ is of genus 3, and $\phi^*m_1 = 2M_1', \phi^*m_2 = 2M_2'$ are
two double fibers. The relation $2(M_1' - M_2') \equiv 0$ gives an
etale double cover $\pi: Y \rightarrow S$. Using formula
\ref{tool}, we obtain the invariants of $Y$: $\chi(Y) = 2, p_g(Y)
= h^0(S, \phi^*e_1 + C' + 2M_1') + p_g(S) = h^0(S, \phi^*l + C') +
1 = h^0(S, 2K_S - \phi^*l - C') + 1 = 4,~q(Y) = 3$. Let $\pi'\circ
g: Y \rightarrow B \rightarrow \mathbb{P}^1$ be the Stein
factorization of the map $f\circ \pi$. Observe that by
construction of $Y$, $\pi'$ is branched over exactly the two points
$f(M_1'),f(M_2')$, hence $B$ is a rational curve. Then applying
Lemma \ref{basiclemma} rules out this case.

$Case~A_2:$ Remark that $L' - C' - E_1' - E_2'$ is nontrivial because otherwise $h^0(S,K_S) = h^0(S, K_S + L' - C' - E_1' - E_2') = h^0(S, 2L' + C') \geq 3$. Let $\pi: Y \rightarrow S$ be the double cover given
by the relation $2(L' - C' - E_1' - E_2') \equiv 0$. Similarly
using formula \ref{tool}, we obtain the invariants of $Y$:
$\chi(Y) = 2, p_g(Y) = h^0(S, 2L' + C') + p_g(S) =h^0(S, \phi^*l +
C') + 1  = h^0(S, 2K_S - \phi^*l - C') + 1 = 4,~q(Y) = 3$. Let $f:
S \rightarrow \mathbb{P}^1$ be the fibration defined by the pencil
$|L_1| = \phi^*|l - e_1|$. Let $\pi'\circ g: Y \rightarrow B \rightarrow
\mathbb{P}^1$ be the Stein factorization of the map $f\circ \pi$.
Lemma \ref{basiclemma} tells that $g(B)\geq 1$, and $\pi'$ is a
double cover branched over at least 4 points, so $f$ has hat least
4 double fibers.
\begin{Claim}The fibration $f$ has at most 4 double
fibers.\end{Claim}
\begin{proof}[Proof of the claim]
Note that $f$ factors through a fibration $\bar{f}: \bar{S}
\rightarrow \mathbb{P}^1$. We show that $\bar{f}$ has at most 4
double fibers. Otherwise we can find 5 double fibers
$2\bar{M_i'}, i = 1,2,3,4,5$. Let $R$ be the ramification divisor
of $\bar{\phi}$. Since every element in $|l - e_1|$ is reduced, we
conclude that $\sum_i\bar{M_i'} \leq R$. By the formula
$K_{\bar{S}} \equiv \bar{\phi}^*K_\Sigma + R$, we get that
$R(\bar{\phi}^*e_1) = K_{\bar{S}}(\bar{\phi}^*e_1) -
(\bar{\phi}^*K_\Sigma) (\bar{\phi}^*e_1) = 6$, and then $(R -
\sum_i\bar{M_i'})(\bar{\phi}^*e_1) = -4 = (\bar{\phi}^*e_1)^2$ since
$\bar{M_i'}(\bar{\phi}^*e_1) = 2$. But this contradicts Lemma
\ref{rmfdivisor}, so the claim is true.
\end{proof}
Therefore, $f$ has exactly 4 double fibers:
$2M_1',2M_2',2M_3',2M_4' = 2(E_2'+C')$, and $\pi'$ is branched over
precisely the 4 points $f(M_i'), i =1,2,3,4$ on $\mathbb{P}^1$,
and $g(B) = 1$. We get the following commutative diagram:
\[\begin{CD}
Y       @>\pi>>       S \\
@VgVV               @VfVV \\
B       @>\pi'>>    \mathbb{P}^1
\end{CD} \]\label{diagram}
We can see $Y$ is in fact the normalization of the fiber product $S
\times_{\mathbb{P}^1}B$. So by Lemma \ref{cvroffbr}, we get that $L' - C' - E_1' - E_2'|_F \equiv 0$ if $F$ is a fiber of $f$ different from $2M_i'$, and $L' - C' - E_1' - E_2'|_{M_i'} \equiv M_i'|_{M_i'}$ for $i = 1,2,3,4$. Then the line bundle $(L' - C' - E_1' - E_2')
- (M_1'+M_2'-M_3'-(E_2'+C'))$ is trivial when restricted to every fiber of $f$, so it is the pull-back $f^*L$ for some line bundle $L$ on $\mathbb{P}^1$. Considering their Chern classes, we can see $L$ is numerically trivial, hence $L \equiv 0$, and then $L' - C' - E_1' - E_2'
\equiv M_1'+M_2'-M_3'-(E_2'+C')$, i.e, $L' +M_3' \equiv M_1'+ M_2'
+ E_1'$. Notice that $2(L' + M_3') \equiv \phi^*(2l - e_1)$, and
take an element $l_2 \in |l-e_2|$, then we have $L' +M_3' + M_1'+
M_2' + E_1' + \phi^*l_2 \in \phi^*|-K_\Sigma|$. So there exists an
element $h \in |-K_\Sigma|$ such that $L' +M_3' + M_1'+ M_2' +
E_1' + \phi^*l_2  = \phi^*h$, thus $l_2 \leq h$ and $L' +M_3' +
M_1'+ M_2' + E_1' = \phi^*(h - l_2)$, hence  $h - l_2$ contains
the curve $\phi(L' +M_3' + M_1'+ M_2'
+ E_1')$. It is impossible since the reduced divisor supported on $\phi(L' +M_3' + M_1'+ M_2'
+ E_1')$ is linear equivalent to $4l - 2e_1$, so $Case~A_2$ does not occur.

$Case~B:$ By the maps $\phi: S \rightarrow \Sigma$ and $\psi:
\hat{P} \rightarrow \Sigma$, we get a surface $\hat{S}$ fitting into the commutative diagram:
\[\begin{CD}
\hat{S}        @>\hat{\phi}>>        \hat{P} \\
@V\mu VV                              @V\psi VV \\
S                @>\phi>>            \Sigma
\end{CD} \]
where $\mu: \hat{S} \rightarrow S$ is composed with some blow-ups.

We have
\begin{equation}\label{ramf}
K_{\hat{S}} \equiv \hat{\phi}^*(K_{\hat{P}}) + R + F
\end{equation}
where $R$ is the ramification divisor and $F$ is an
exceptional divisor of $\hat{\phi}$, and
\begin{equation}\label{pullback}
2K_{\hat{S}} \equiv \mu^*(2K_S) + 2E \equiv \mu^*\phi^*(-K_\Sigma)
+ 2E \equiv \hat{\phi}^*(-K_{\hat{P}}) + 2E
\end{equation}
where $E$ is zero divisor or an effective $\mu$-exceptional divisor.

By $-K_\Sigma \equiv 2l + c$ and Proposition \ref{c}, we can write $2K_S \equiv 2L + 2C'$, and then get the relation $2(K_S -L -C') \equiv 0$ which gives an etale double cover
$\pi: Y \rightarrow S$. Using Formula \ref{tool}, we calculate the invariants of $Y$: $\chi(Y) = 2,
p_g(Y) = h^0(S, 2K_S - L - C') + 1 = 4,~q(Y) = 3$. By base extension,
we get an etale double cover $\hat{\pi}: \hat{Y} \rightarrow
\hat{S}$ over $\hat{S}$.

Let $f: \hat{S} \rightarrow \mathbb{P}^1$ be the fibration induced
by the pencil $\hat{\phi}^*|l- e_1 - e_2|$. Consider the Stein
factorization of the map $f \circ \hat{\pi}$. Then applying Lemma
\ref{basiclemma} and similar argument as in $Case~ A_2$, we show
that the fibration $f$ has at least 4 double fibers. Select 4
double fibers and denote them by $2M_i', i =1,2,3,4$.

Since $E \geq 0$ and $\hat{\phi}_*E$ is supported on $e_1$, calculating the intersection number of
equation $\ref{pullback}$ and $(\hat{\phi}^*e_1)$ yields
$$2K_{\hat{S}}(\hat{\phi}^*e_1) = (\hat{\phi}^*(-K_{\hat{P}}) + 2E)(\hat{\phi}^*e_1)  = 2(\hat{\phi}_*E) \cdot e_1 \leq 0$$

Since $F$ is $\hat{\phi}$-exceptional, calculating the intersection number of equation
$\ref{ramf}$ and $\hat{\phi}^*e_1$ yields
$$K_{\hat{S}}(\hat{\phi}^*e_1) =(\hat{\phi}^*(K_{\hat{P}}) + R + F)(\hat{\phi}^*e_1) = R(\hat{\phi}^*e_1)$$
thus $R(\hat{\phi}^*e_1) \leq 0$.

Since every element in $|l - e_1 - e_2|$ is reduced, there exists
a $\hat{\phi}$-exceptional divisor $F'$ such that $\sum_iM_i' - F' \leq R$. Since
$M_i'(\hat{\phi}^*e_1) = 2$ and $F'(\phi^*e_1) = 0$, it follows
that $(\sum_iM_i' - F')(\hat{\phi}^*e_1) = 8$. In turn we get that
$$(R - (\sum_iM_i' - F'))\hat{\phi}^*e_1 \leq -8 =
(\hat{\phi}^*e_1)^2$$ However, this contradicts Lemma
\ref{rmfdivisor}, and $Case~B$ does not occur.

In conclusion, we proved that $d\neq 4$.
\end{proof}

\begin{Remark}\label{bor}
Applying the method above to the case $K^2 = 8$, we obtain an alternative proof that the degree of the bicanonical map $d \neq 4$ which is the main result of \cite{Bor2},
here we give a sketch. If the bicanonical image $\Sigma$ is the the Veronese embedding in $\mathbb{P}^8$ of a quadric $Q$ in $\mathbb{P}^3$, then arguing as in Section \ref{KS}, we can rule out this case by Claim \ref{coefficient}
and condition \ref{condition}; if $\Sigma$ is a del Pezzo surface of degree 8, once writing out the effective divisor linearly equivalent to $K_S$, we can find three double fibers, and then rule out this case as in Section \ref{pf} Case $A_1(i)$.
\end{Remark}

\section{the case when $d=2$}\label{deq2}
In this section, we consider the case the bicanonical map is of degree 2 and aim to prove Theorem \ref{main}

\subsection{The notations and known results}
In this section, let $S$ a smooth minimal surface of general type with $K_S^2
= 7, p_g(S) = q(S) = 1$, and assume its bicanonical map $\phi$ is of degree 2. Then it induces an involution $\sigma$ on $S$. The fixed locus of $\sigma$ is the union of a smooth curve $R_{\sigma}$ and $t$ isolated points,
thus the singularities of $S/(\sigma)$ consist of exactly $t$ nodes. Denote by $\eta: S \rightarrow S/(\sigma)$ the projection onto the quotient, by $q: V \rightarrow S$ be the blow-up at the $t$ isolated fixed points
of $\sigma$, by $p: W \rightarrow S/(\sigma)$ the minimal resolution of $S/(\sigma)$. Then there is a natural involution $i$ on $V$ such that $W = V/(i)$, and the quotient map $\pi: V \rightarrow W$ is a double cover. In all, we get the following commutative diagram:
\[\begin{CD}
V       @>q>>       S \\
@V\pi VV               @V\eta VV \\
W       @>p>>    S / \sigma
\end{CD} \]
Denote by $B''$ be the branch divisor of the map $\eta: S \rightarrow
S/(\sigma)$, and put $B' = p^*B''$. Then $\pi: V \rightarrow W$ is a double cover branched along $B= B' + \sum_{i=1}^{i=t}C_i$, so there exists a line bundle $L$ on $W$ such that $2L \equiv B$.
Let $\rho: W \rightarrow P$ be the map to its
minimal model, and put $\bar{B} =
\rho_*B$.
Now we collect some known results in a proposition, some of which appeared in several papers such as \cite{CM}, \cite{Bor1}, and we refer the readers to \cite{Rito1} Theorem 1 and section 2 for all the details.

\begin{Proposition}\label{tool2}Let all the notations and assumptions be as above. Then
\begin{enumerate}
\item[i)]{the bicanonical map $\phi$
factors through $\eta$ if and only if $h^0(W,
2K_W + L) = 0$;}
\item[ii)]{$t = K_S^2 - 2\chi(S) + 6\chi(W) + h^0(W,
2K_W + L)$;}
\item[iii)]{if $P$ is not rational, then it is a K3 surface, and $\bar{B}$ has a 4-uple or [3,3]-point and possibly negligible
singularities.}
\end{enumerate}
\end{Proposition}

So the surface $P$ is either a $K3$ surface or a rational surface.

\subsection{Case $I$: $P$ is a $K3$ surface.}

By Proposition \ref{tool2}, the
number of the isolated fixed point $t = K_S^2 + 6\chi(W) -
2\chi(S) = 17$. Since a $K3$ surface
has at most 16 disjoint $(-2)$-curves, at least one $(-2)$-curve
is contracted by $\rho$. Assume $C = C_{17}$ is contracted. Note that
since $h^{1,1}(W) \leq h^{1,1}(V) = 22$ and $h^{1,1}(P) = 20$, the
map $\rho$ is composed with at most two blow-ups. Since a $(-2)$-curve is contracted, $\rho$ is the blow-up at two points infinitely near. Then we conclude that $\bar{B}$ has a [3,3]-point $Q$ as its unique singularity, and $\rho$ is the blow-up at $Q$ and another point infinitely
near to $Q$. Therefore $P$ has exactly 16 disjoint $(-2)$-curves,
thus it is a Kummer surface by \cite{Nik}.

Put $D = \rho_*B'$, denote by $E$ the exceptional $(-1)$-curve on $W$. It follows that $K_W \equiv 2E + C$ and $B' = \rho^*D - 3C - 6E$. We denote by $C_i$ the $(-2)$-curve $\rho_*C_i$ if there is no confusion. Note that there exists a line bundle $L_1$ on $P$ such that $\sum_{i=1}^{i=16}C_i \equiv 2L_1$. So on $W$, we have $\sum_{i=1}^{i=16}C_i \equiv 2\rho^*L_1$. Then by $2L \equiv B = \rho^*D - 2C - 6E + \sum_{i=1}^{i=16}C_i$, we conclude that there exists a line bundle $L'$ on $P$
such that $D \equiv 2L'$.

Since $B'C = 0$, it follows that $K_W B' = 2EB'
= 6$. By $L \equiv_{num} \frac{B' +
\sum_{i=1}^{i=17}C_i}{2}$, we have $K_SL = 3$. Then by
$\chi(V) = 2\chi(W) + \frac{(K_W+L)L}{2}$, we get that $L^2 = -9$,
then $B'^2 = -2$, hence $D^2 = 16$ since $\rho^*D = B' + 3C +
6E$.

Note that the linear system $|L'|$ induces a morphism $\varphi:P \rightarrow \mathbb{P}^3$, and $\varphi$ is birational or of degree 2 (cf. \cite{Beau2} Prop. VIII.13). Let $v: P \rightarrow \bar{P}$ be the map contracting the 16 disjoint $(-2)$-curves. Since $P$ is an Kummer surface, it is the quotient of an Abelian variety $A$. Let $u: A \rightarrow \bar{P}$ be the quotient map. Note that the line bundle $L'$ descends to a line bundle $\bar{L'}$ on $\bar{P}$. Obviously the line bundle $u^*\bar{L'}$ on $A$ is ample, so is $\bar{L'}$ on $\bar{P}$. Then we can see that the map $\varphi$ contracts exactly the 16 disjoint $(-2)$-curves, and it factors through a finite morphism $\bar{\varphi}: \bar{P}\rightarrow \mathbb{P}^3$.

If $\varphi$ is birational, then $\bar{\varphi}$ maps $\bar{P}$ isomorphically to a quartic surface with 16 nodes in $\mathbb{P}^3$, and Case $I(i)$ in Theorem \ref{main} follows.

If $\varphi$ is of degree 2, then $\bar{\varphi}$ is a double cover over a surface of degree 2 in $\mathbb{P}^3$: either a quadric cone or a Segre embedding of $\mathbb{P}^1\times \mathbb{P}^1$.

\begin{Claim}
The image of $\varphi$ can not be a quadric cone.
\end{Claim}
\begin{proof}[Proof of the claim]
To the contrary, assume that $\varphi$ is of degree 2, and its image $X$ is a quadric cone. The morphism induces an involution $\tau$ on $P$ since $P$ is minimal. Denote by $\nu: P \rightarrow P/ \tau$ the quotient map.
Note that $\varphi$ factors through a map $\mu: P/ \tau \rightarrow X$, and $\mu$ contracts exactly the image of the 16 disjoint $(-2)$-curves.

Denote by $Q$ the unique node on $X$. We claim that the inverse image $\mu^{-1}(Q)$ is either a node or a $(-2)$-curve $C$ such that $\nu^*C$ is composed with two disjoint $(-2)$-curves. Indeed, because otherwise $v^*C$ must be a $(-2)$-curve, and thus $C^2= -1$, it is impossible since $C$ is contracted to a node no matter whether $C$ contains the nodes on $P/\tau$ or not.

Consider the branch divisor $D$ of $\nu$. We can see $D \cap \mu^{-1}Q = \phi$, thus $\mu_*D$ does not contain $Q$. The double cover $\bar{\varphi}: \bar{P} \rightarrow X$ is branched over $\mu_*D$ and possibly the point $Q$. Since $K_X \equiv -2H$ where $H$ is a hyperplane on $X$, we conclude that the divisor $\mu_*D \equiv_{num} 4H$, moreover it has at least 14 nodes since $\bar{P}$ contains 16 nodes.

Let $\epsilon: Y \rightarrow X$ be the minimal resolution of $X$. Then $Y$ is the Hirzebruch surface $\mathbb{F}_2$, and $\epsilon^*H = 2\Gamma + C_0$ where $|\Gamma|$ is the ruling of $\mathbb{F}_2$ and $C_0$ is the section with self-intersection -2. Then we have $\epsilon^*(\mu_*D) \equiv_{num} 8\Gamma + 4C_0$ and $(K_Y + \epsilon^*(\mu_*D))\epsilon^*(\mu_*D) = 16$, thus the arithmetic genus $p_a(\epsilon^*(\mu_*D)) = 9$. The divisor $\epsilon^*(\mu_*D)$, which is mapped isomorphically to $\mu_*D$, also has at least 14 nodes. If we denote by $n$ the number of the irreducible components of $\epsilon^*(\mu_*D)$, then we have $14 - n +1 \geq 9$, so $\epsilon^*(\mu_*D)$ contains at least 6 irreducible components, and at least one component is a fiber $\Gamma_0$ due to $\epsilon^*(\nu_*D)\Gamma = 4$. Then since $(\epsilon^*(\mu_*D) - \Gamma_0)C_0 < 0$, $C_0$ is contained in $\epsilon^*(\nu_*D) - \Gamma_0$, and a contradiction follows from the fact that $\mu_*D$ does not contain $Q$.
\end{proof}

So if $\varphi$ is of degree 2, it factors through the double cover $\bar{\varphi}: \bar{P} \rightarrow \mathbb{P}^1\times \mathbb{P}^1$. Note that the branch divisor is a $(4,4)$ divisor with exactly 16 nodes, hence it composed with 4 vertical curves and 4 horizontal curves. What's more, $P$ is a Kummer surface obtained by an Abelian variety isomorphic to $E_1 \times E_2$ where $E_1,E_2$ are two elliptic curves. So we get Case $I(ii)$ in Theorem \ref{main}.

Let's begin another direction. Conversely, given a Kummer surface $P$ and a divisor $D + \sum_{i=1}^{i=16}C_i$ as in Theorem \ref{main} $I)$, inverting the process above, we get a surface $S$ with $K_S^2 =  7$ and $\chi(S)=1$. We still need to show that $p_g(S) = 1$ and its bicanonical map is not birational. We assume that $D \equiv 2L'$ and $\sum_{i=1}^{i=16}C_i \equiv 2L_1$, then we have $L \equiv \rho^*L' - 3E - C + \rho^*L_1$ and $K_W + L \equiv \rho^*L' - E + \rho^*L_1$.

\begin{Claim}$h^0(W, \rho^*L' - E + \rho^*L_1) = 0$. \end{Claim}
\begin{proof}[Proof of the claim]
Since $(\rho^*L' - E + \rho^*L_1)C_i < 0$ for $i =1,2,...,17$, we have $h^0(W, \rho^*L' - E + \rho^*L_1) = h^0(W, \rho^*L' - E  + \rho^*L_1 - C - \sum_{i=1}^{i=16}C_i)= h^0(W, \rho^*L' - E -C - \rho^*L_1)$. It suffices to show that $h^0(W, \rho^*L' - \rho^*L_1)  = 0$, equivalently $h^0(P, L'- L_1) = 0$. Since $P$ is a Kummer surface, we get the following commutative diagram:
\[\begin{CD}
\tilde{A}       @>v'>>       A\\
@Vu' VV               @VuVV \\
P      @>v>>    \bar{P} \\
\end{CD} \]
where $A$ is a principally polarized Abelian variety (either a Jocobian of a curve of genus 2 or $E_1 \times E_2$ where $E_1,E_2$ are two elliptic curves), $u$ is the quotient map induced by the involution $(-1)_A$, and $v,v',u',\tilde{A}$ are as usual. Note that $u^*(v_*L') \equiv 2\Theta$ where $\Theta$ is a theta divisor on $A$. If we denote by $E_i,i=1,...,16$ the 16 disjoint $(-1)$-curves on $\tilde{A}$, then $u'^*L_1 \equiv \sum_{i=1}^{i=16}E_i$. Considering the commutative diagram above, if $h^0(P, L' - L_1)  > 0$, then there exists a divisor $D$ on $A$ such that $D\equiv 2\Theta$ and $D$ passes through the 16 fixed points of $(-1)_A$. If $A$ is the Jocobian of a curve of genus 2, we can find a theta divisor $\Theta_0$ such that it is not contained in $D$ and passes through 6 points of the 16 fixed points, thus $\Theta_0D\geq 6$ which contradicts $\Theta_0D = 4$. If $A = E_1 \times E_2$, we can find some curve $C = E_1 \times p$ where $p$ is one of the fixed point of the involution $(-1)_{E_2}$ such that it is not contained in $D$, thus $CD\geq 4$ which contradicts $CD = 2$. Therefore, the claim follows.
\end{proof}
So we have $p_g(V) = p_g(W) + h^0(W, K_W + L) = 1$. Similarly, we can show $h^0(W, 2K_W + L) = 0$, this means the bicanonical map of $S$ factors through a degree 2 map.

\subsection{Case $II$: $P$ is a rational surface.}

By \cite{Bor1} Theorem 0.4, $S$ is a smooth minimal model of a Du Val double plane. Recall the definition of Du Val double plane in \cite{Bor1}.

\begin{Definition}\label{DV}
A Du Val double plane $Y$ (of type $\mathcal{B}, \mathcal{D}, \mathcal{D}_n$) is given by a double cover $Y \rightarrow X$ branched over a reduced curve $G$ on $X$ where $X$ and $G$ satisfy one of the following:
\begin{enumerate}
\item[$\mathcal{D}$)]{$X = \mathbb{P}^2$ and $G$ is a smooth curve of degree 8;}
\item[$\mathcal{D}_n$)]{$X = \mathbb{P}^2$ and $G = G_0$ or $G = G_n + L_1 + \cdot\cdot\cdot + L_n$ where $G_n$ is a curve of degree $10+n$ and the $L_i$'s are distinct lines through a point $\gamma$. The essential singularities of $G$ are as following:
    \begin{enumerate}
    \item[-]{$\gamma$ is a singular point of multiplicity $2n+ 2$,}
    \item[-]{$G$ has a [5,5]-point on $L_i,i=1,2,...,n$,}
    \item[-]{possibly some [4]-points and [3,3]-points.}
    \end{enumerate}}
\item[$\mathcal{B}$)]{$X$ is the Hirzebruch surface $\mathbb{F}_2$ and $G = C_0 + G'$ where $G' \in |7C_0 + 14\Gamma|$ (here $|\Gamma|$ is the ruling of $\mathbb{F}_2$ and $C_0$ is the section with self-intersection -2). The only possible essential singularities of $G$ are [3,3]-points such that if $[q' \rightarrow q]$ is one of them with $q \in \mathbb{F}_2$, then the fiber through $q$ is tangent to $G$ at $q$.}
\end{enumerate}
\end{Definition}
Denote by $b_3$ ($b_4$) the number of the [3,3]-points ([4]-points) on $G$.
Since $K_S^2 = 7$ and $\chi(S) = 1$, applying \cite{Bor1} Proposition 4.7, 4.14 and 4.15, we get that
\begin{enumerate}
\item[i)] {$S$ arises from a minimal resolution of a Du Val double plane of type $\mathcal{D}_5$;}
\item[ii)] {$b_3 = 1$, $b_4 = 0$, and there exists exactly one conic containing the points: $p_i,i=1,2,3,4,5$ and $q_1$ where $p_i$ is the [5,5]-point on $L_i$ and $q_1$ is the [3,3]-point on $G$;}
\item[iii)] {there exists hyperelliptic genus 3 fibration on $S$ with 5 double fibers, and the bicanonical map induces the hyperelliptic involution on a general fiber.}
\end{enumerate}

Conversely, let $S$ be the minimal resolution of a Du Val double plane of type $\mathcal{D}_5$ described as in Definition \ref{DV} satisfying the condition $ii)$. By Proposition 4.7 and 4.9  in \cite{Bor1}, $ii)$ means $p_g(S) = 1$, $K_S^2 = 7$ and $\chi(S) = 1$; and by Proposition 4.2 in \cite{Bor1}, the bicanonical map is not birational.

At last, we complete Theorem \ref{main}.

\section{The existence of surfaces of type $(I)$}\label{K3}
Let all the notations and assumptions be as in the previous section.
\subsection{Fibration structure}\label{fib}
Denote by $\alpha: S \rightarrow C$ the Albanese pencil. The
involution $\sigma$ induces an action $\sigma^*$ on $H^0(S,
\Omega^1)$, which is multiplying $-1$ on every element in $H^0(S,
\Omega^1)$, so $\sigma$ maps a fiber of $f$ to another. Therefore we
get a fibration $\alpha': S / \sigma \rightarrow \mathbb{P}^1$ and a
double cover $\eta': C \rightarrow \mathbb{P}^1$ which fit into the
following commutative diagram:
\[\begin{CD}
V       @>q>>       S      @>\alpha >>   C \\
@V\pi VV               @V\eta VV          @V\eta' VV      \\
W       @>p>>    S / \sigma  @>\alpha' >> \mathbb{P}^1
\end{CD} \]
Thus we get a fibration of $f: W \rightarrow \mathbb{P}^1$. Denote
by $P_i,1=1,2,3,4$ the four branch point of $\eta'$, and put $F_i =
f^*P_i$. Then every component of the branch locus of $\pi$ is contained in
one of $F_i,i=1,2,3,4$.
\subsection{Case $I(i)$}\label{I(i)} In this section, let $\bar{P}$ be a Kummer surface embedded in $\mathbb{P}^3$, and assume there exists a quadric $Q$ such that $\bar{B'} = Q \cap \bar{P}$ has a [3,3]-point as its unique singularity and does not pass through any node on $\bar{P}$. Then we get a surface $S$ described in Theorem \ref{main}.

First we have
\begin{Claim}\label{pic}
The Picard number of $\bar{P}$ is $\geq 2$, consequently a general Kummer surface embedded in $\mathbb{P}^3$ can not be
birational to the bicanonical image of $S$.
\end{Claim}
\begin{proof}
Otherwise if we denote by $NS(W) \subset H^2(W, \mathbb{Z})$ the Neron-Severi group of $W$, then $NS(W)\otimes \mathbb{Q}$ is spanned by $H', C, E,
C_i, i=1,...,16$ where $H'$ is the pull-back of a hyperplane via the map $W \rightarrow \bar{P} \rightarrow \mathbb{P}^3$. A general fiber $F$ of $f$ is numerically
equivalent to $k(aH' - C - 2E)$ where $k, a \in \mathbb{Q}$ since
$FC_i = 0, i=1,...,17$. Note that $H'^2 = 4$. Then a contradiction follows since the
equation $(aH' -C- 2E)^2 = 4a^2 - 2 = 0$ in $a$ has no rational
solution.
\end{proof}

Let $H$ be the hyperplane section on $Q$.
If we see $\bar{B'}$ as a curve on $Q$, then
$\bar{B'} \equiv 4H$ on $Q$ and has a [3,3]-point as its unique
singularity.

\begin{Claim}
$Q$ is normal.
\end{Claim}
\begin{proof}
First note that $Q$ is reduced since $\bar{B'}$ is reduced. So if $Q$ is not normal, then $Q = L_1\cup L_2$ where $L_1,L_2$ are two hyperplanes. By abuse of notations, we write $\bar{B'} = L_1 + L_2$ where $L_1,L_2$ are two divisors on $\bar{P}$. Since $L_1+ L_2$ are smooth except for a [3,3]-point and $L_1L_2 = 4$, so $L_1\cap L_2$ coincides with the [3,3]-point on $\bar{B'}$, and we conclude that one of $L_i,i=1,2$ is smooth. If we assume $L_1$ is smooth, then its strict transformation $\tilde{L_1}$ on $W$ is linearly equivalent to $\rho^*L_1 - C -2E$, thus $\tilde{L_1}^2 = 2$ which contradicts the fact that it is contained in one fiber of $f$ (cf. Section \ref{fib}).
\end{proof}

From now on, we fix $Q$ and assume that it is normal. Now we begin to construct a family of quartics in $\mathbb{P}^3$.
Let $p$ be a point on $Q$, $q$ a point infinitely near $p$. We define the following sets.
\begin{enumerate}
\item[$\mathcal{D}_{p,q}$:]{$\mathcal{D}_{p,q} = \{D \in |4H||p,q \in D~and~ p ~is~ a ~[3,3]-point~on~D\}$;}
\item[$\mathcal{F}_{p,q}$:]{the Hilbert scheme of the quartics such that the intersection with $Q$ is a curve in $\mathcal{D}_{p,q}$;}
\item[$\mathcal{F}_{p,q}^0$:]{$\mathcal{F}^0 \subset \mathcal{F}_{p,q}$ be the subset consisting of the ``good'' elements where an element $F \in \mathcal{F}$ is said ``good'' if $F$ has at most nodes, and $F\cap Q$ has a [3,3]-point as its unique singularity  and does not pass through any nodes on $F$;}
\item[$\mathbb{P}^{34}$:]{the Hilbert scheme of the quartics in $\mathbb{P}^3$;}
\item[$\mathcal{K}$:]{the Hilbert scheme of the Kummer surfaces in $\mathbb{P}^3$.}
\end{enumerate}

So the existence of a surface of type $I(i)$ is equivalent to the following problem
\begin{Problem}\label{reducedproblem}
Is there a quadric $Q$ and $p,q \in Q$ as above such that $\mathcal{K} \cap \mathcal{F}_{p,q}^0 = \phi$?
\end{Problem}

Let $\sigma: \tilde{Q} \rightarrow Q$ be the blow-up at $p$ and the point $q$ infinitely near $p$, and denote by $C$ the exceptional $(-2)$-curve and by $E$ the $(-1)$-curve over $q$.
\begin{Proposition}
Let the notations be as above. Then we have
\begin{enumerate}
\item[(i)]{$\mathcal{D}_{p,q}$ can be seen as an open set of $\mathbb{P}^{dim|\sigma^*4H - 3C - 6E|}$, and $dim(\mathcal{D}_{p,q}) = 13$ (resp. $12$) if $q$ is (resp. not) a direction of a line;}
\item[(ii)]{$\mathcal{F}_{p,q}$ is a rank 10 vector bundle over $\mathcal{D}$;}
\item[(iii)]{$\mathcal{F}_{p,q}^0$ is an open set of $\mathcal{F}_{p,q}$, moreover it is an open set of a projective variety.}
\end{enumerate}
\end{Proposition}
\begin{proof}
$(i)$ Note that $\mathbb{P}^{dim|\sigma^*4H - 3C - 6E|}$ can be seen as the Hilbert scheme of the curves on $Q$ that has a ``bad'' enough singularity at $p$. The condition that a curve has a ``worse'' singularity than a [3,3]-point at $p$ or other singularities other than $p$ is a closed condition, so $\mathcal{D}_{p,q}$ can be seen as an open set of $\mathbb{P}^{dim|\sigma^*4H - 3C - 6E|}$.

Now let's calculate $dim|\sigma^*4H - 3C - 6E|$. For simplicity, we assume $Q$ is smooth, i.e., $Q \simeq \mathbb{P}^1\times \mathbb{P}^1$. By Riemann-Roch and $K_{\tilde{Q}}\equiv -\sigma^*2H + C + 2E$, we have:
$$\chi(\tilde{Q}, \sigma^*4H - 3C - 6E) = 13$$
If $q$ is not a direction of a line, there is an effective irreducible divisor $D \equiv \sigma^*H - C - 2E$, so we can see that $\sigma^*4H - 3C - 6E - K_{\tilde{Q}} \equiv \sigma^*6H - 4C - 8E \equiv 2H + 4D$ is nef and big. By Kawamata-Viehweg vanishing theorem, we have $h^i(\tilde{Q}, \sigma^*4H - 3C - 6E) = h^i(\tilde{Q}, K_{\tilde{Q}} + (\sigma^*4H - 3C - 6E - K_{\tilde{Q}})) = 0, ~i=1,2$, and thus $dim(|\sigma^*4H - 3C - 6E|) = 12$.

If $q$ is a direction of a line $l$, then there exists an effective divisor $F \equiv \sigma^*l - C - 2E$. Note that $F$ is the fixed part of $|\sigma^*4H - 3C - 6E|$ since $(\sigma^*4H - 3C - 6E)F < 0$. So
$h^0(\tilde{Q}, \sigma^*4H - 3C - 6E) = h^0(\tilde{Q}, \sigma^*4H - 3C - 6E - F)$.  Similarly checking that $\sigma^*4H - 3C - 6E - F$ is nef and big and using Kawamata-Viehweg vanishing theorem, we get $h^0(\tilde{Q}, \sigma^*4H - 3C - 6E - F) = \chi(\tilde{Q}, \sigma^*4H - 3C - 6E - F) = 14$, so we are done.

$(ii)$ Consider the exact sequence
$$0 \rightarrow \mathcal{O}_{\mathbb{P}^3}(2) \rightarrow \mathcal{O}_{\mathbb{P}^3}(4) \rightarrow \mathcal{O}_{Q}(2) \rightarrow 0$$
We get the following exact sequence
$$0 \rightarrow H^0(\mathcal{O}_{\mathbb{P}^3}(2)) \rightarrow H^0(\mathcal{O}_{\mathbb{P}^3}(4)) \rightarrow H^0(\mathcal{O}_{Q}(2)) \rightarrow 0$$
And $(ii)$ follows.

$(iii)$ Note that the ``bad'' conditions--- the curve $F\cap Q$ has other singularities except for a [3,3]-point or passes though a node on $F$---are closed. So $\mathcal{F}_{p,q}^0$ is an open set in $\mathcal{F}_{p,q}$. The remaining statement is an easy consequence of the definition and the fact that $\mathcal{D}_{p,q}$ can be compactified as $\mathbb{P}^{dim|\sigma^*4H - 3C - 6E|}$.
\end{proof}

No doubt the dimension of the image via the natural map $PGL(3) \times \mathcal{F}^0 \rightarrow \mathbb{P}^{34}$ can help us understand Problem \ref{reducedproblem}. Consider the action of $PGL(3)$, we have
\begin{Fact}\label{gp}
With the notations above, we have
\begin{enumerate}
\item[(i)]{Let $F \in \mathcal{F}$ be general, and let $\sigma \in PGL(3)$. If $\sigma(F) \in \mathcal{F}$, then $\sigma$ preserves the quadric $Q$.}
\item[(ii)]{Let $G_Q$ be the subgroup of $PGL(3)$ that preserves the quadric $Q$ and fixes the points $p$ and $q$. Then $dim(G_Q) = 4$ (resp. 3) if $Q$ is smooth and $q$ is (resp. not) a direction of a line, and $dim(G_Q) = 5$ (resp. 6) if $Q$ is a cone and $q$ is (resp. not) a direction of a line;}
\end{enumerate}
\end{Fact}
\begin{proof}
$(i)$ Assume $F \cap Q = D$. Then by definition, $\sigma(D)$ is also contained in $Q$ since $\sigma(F) \in \mathcal{F}$, hence $\sigma(D) \subset Q \cap \sigma(Q)$. So it follows that $Q = \sigma(Q)$ since $\sigma(D) \equiv 4H$ on $Q$.

$(ii)$ Considering the quadric form defining $Q$, we know that the subgroup of $PGL(3)$ preserving $Q$ is of dimension 6 if $Q$ is smooth, and is of dimension 7 if $Q$ is a cone. Then noticing that every element in $PGL(3)$ maps a line to a line, we are done.
\end{proof}

\begin{Proposition}
The dimension of the image via the natural map $PGL(3) \times \mathcal{F}^0 \rightarrow \mathbb{P}^{34}$ is 34 if $Q$ is smooth, and is 33 if $Q$ is a cone.
\end{Proposition}
\begin{proof}
Note that the dimension of orbit $dim(O_F \cap \mathcal{F}_{p,q}^0) = dim(G_Q)$ for a general $F \in \mathcal{F}_{p,q}^0$ and that the dimension of the image is $dim(\mathcal{F}_{p,q}) + dim(PGL(3)) - dim(O_F \cap \mathcal{F}_{p,q}^0)$. Then the proposition follows.
\end{proof}

As a corollary we have
\begin{Corollary}
Let the notations be as at the beginning of Section \ref{I(i)}. Then $Q$ must be a cone.
\end{Corollary}
\begin{proof}
We only need to exclude the case $Q$ is smooth. Remark that we can embed $PGL(3) \times \mathcal{F}^0$ to a projective variety as an Zariski open set. If $Q$ is smooth, then the image of the natural map $PGL(3) \times \mathcal{F}^0 \rightarrow \mathbb{P}^{34}$ is of dimension 34 and thus open in $\mathbb{P}^{34}$.
Considering Problem \ref{reducedproblem} instead, Claim \ref{pic} tells that a general point of $\mathcal{K}$ is contained in the closed set $\mathbb{P}^{34} - PGL(3)\mathcal{F}^0$, thus $\mathcal{K} \cap PGL(3)\mathcal{F}^0 =\phi$, and we are done.
\end{proof}

\subsection{The case $I(ii)$}
For this case we have the following theorem.
\begin{Theorem}
If $\bar{P}$ is a double cover of $\eta: \mathbb{P}^1\times \mathbb{P}^1$, then the divisor $\bar{B}$ is not the pull-back of a divisor on $\mathbb{P}^1\times \mathbb{P}^1$.
\end{Theorem}
\begin{proof}
To the contrary, we assume $\bar{B} = \eta^*C$ where $C \equiv (2,2)$ and does not pass through any nodes of the branch divisor on $\mathbb{P}^1\times \mathbb{P}^1$. Denote by $p \in \bar{B}$ the [3,3]-point, and let $q = \eta(p)$.
We can see that $q$ is a simple triple point on $C$ if $q$ is not on the branch locus, and is a [3,3]-point otherwise. Either case does not occur since $C \equiv (2,2)$.
\end{proof}

\end{document}